\documentclass[11pt,twoside]{amsart}
\usepackage[russian,english]{babel}
\Russian

\usepackage{amssymb}
\usepackage{latexsym}
\usepackage{calligra}

 \usepackage{mathrsfs}
\usepackage{color}

\catcode`\@=11
\def\omathop#1#2#3{\let\temp=#1\def\letter{#2}
  \ifcat#3_ \let\next\@@olim\else\let\next\@olim\fi\next#3}
\def\@olim{\letter\text{-}\!\temp}
\def\@@olim_#1{\mathchoice{
   \setbox0=\hbox{$\displaystyle\letter\text{-}\!\temp\!\text{-}\letter$}
   \setbox2=\hbox{$\displaystyle\temp$}
   \setbox4=\hbox{$\scriptstyle#1$}
   \dimen@=\wd4 \advance\dimen@ by -\wd2 \divide\dimen@ by2
   \def\next{\letter\text{-}\!\temp_{\hbox to 0pt{\hss$\scriptstyle#1$\hss}}
     \hskip\dimen@}
   \ifdim\wd2>\wd4 \def\next{\@olim_{#1}}\fi
   \ifdim\wd4>\wd0 \def\next{\mathop{\llap{$\letter$-}\!\temp}\limits_{#1}}\fi
   \next}
   {\@olim_{#1}}{\@olim_{#1}}{\@olim_{#1}}}

\def\olim{\omathop{\lim}{o}}

\newcommand{\be}{\begin{equation}}
\newcommand{\ee}{\end{equation}}

\def\phi{\varphi}
\newcommand{\bi}{\begin{itemize}}
\newcommand{\ei}{\end{itemize}}
\newcommand{\bn}{\begin{enumerate}}
\newcommand{\en}{\end{enumerate}}
\newcommand{\hide}[1]{}          

\theoremstyle{plain}
\newtheorem{thm}{Theorem}[section]

\newtheorem{lemma}[thm]{Lemma}
\newtheorem{rem}[thm]{Remark}

\newtheorem{example}[thm]{Example}

\theoremstyle{definition}
\newtheorem{definition}[thm]{Definition}

\numberwithin{equation}{section}

\begin{document}

\title{Domination problem for  narrow  orthogonally additive operators}

\author{Marat~Pliev}

\address{South Mathematical Institute of the Russian Academy of Sciences\\
str. Markusa 22,
Vladikavkaz, 362027 Russia}

\email{maratpliev@gmail.com}




\begin{abstract}
The ``Up-and-down'' theorem which describes the structure of the Boolean algebra of fragments of a linear positive  operator is the   well known result of the operator theory. We prove an analog of this  theorem for a positive abstract Uryson operator defined on a vector lattice and  taking values in a Dedekind complete vector lattice. This result we apply to prove that for an order narrow positive abstract Uryson operator $T$ from a vector lattice $E$ to a Dedekind complete vector lattice $F$, every abstract Uryson operator $S:E\to F$, such that   $0\leq S\leq T$ is  also order narrow.
\end{abstract}

\maketitle


\section{Introduction}

Today the theory of narrow operators is an  active  area   of Functional Analysis (see the recent monograph \cite{PR}). Lately  the concept of the narrowness  was generalized to the setting of orthogonally additive  operators in vector lattices \cite{PP}.  The aim of this article  is  to continue  the  investigation of order  narrow orthogonally additive operators and to consider the domination problem for this class of operators.\footnote{ The research was supported by  Russian Foundation of Fundamental Research,the grant number 14-01-91339}

\section{Preliminaries}

The  goal of this section is to introduce some basic definitions and facts. General information on vector lattices and Boolean algebras  the reader can find in the books \cite{Al,Jec,Ku,LZ}.

Let $E$ be a vector lattice. A net $(x_\alpha)_{\alpha \in \Lambda}$ in $E$ \textit{order converges} to an element $x \in E$ (notation $x_\alpha \stackrel{\rm (o)}{\longrightarrow} x$) if there exists a net $(u_\alpha)_{\alpha \in \Lambda}$ in $E_{+}$ such that $u_\alpha \downarrow 0$ and $|x_\beta - x| \leq u_\beta$ for all $\beta\in \Lambda$. The equality $x=\bigsqcup\limits_{i=1}^nx_i$ means that $x=\sum\limits_{i=1}^nx_i$ and $x_i \bot x_j$ if $i\neq j$.
An element $y$ of  $E$ is called a \textit{fragment} (in another terminology, a \textit{component}) of an element $x \in E$, provided $y \bot (x-y)$. The notation $y \sqsubseteq x$ means that $y$ is a fragment of $x$. Two fragments $x_{1},x_{2}$ of $x$  are called \textit{mutually complemented} or $MC$, in short, if $x = x_1 \sqcup x_2$.  If $E$ is a vector lattice and $e \in E$ then by $\mathcal{F}_e$ we denote the set of all fragments of $e$.

An element $e$ of a vector lattice $E$ is called a \textit{projection element} if the band  generated by $e$ is a projection band. A vector lattice $E$ is said to have the \textit{principal projection property} if every element of $E$ is a projection element. For instance, every Dedekind $\sigma$-complete vector lattice has the principal projection property.

\begin{definition} \label{def:ddmjf0}
Let $E$ be a vector lattice, and let $F$ be a real linear space. An operator $T:E\rightarrow F$ is called \textit{orthogonally additive} if $T(x+y)=T(x)+T(y)$ whenever $x,y\in E$ are disjoint.
\end{definition}

It follows from the definition that $T(0)=0$. It is immediate that the set of all orthogonally additive operators is a real vector space with respect to the natural linear operations.

\begin{definition}
Let $E$ and $F$ be vector lattices. An orthogonally additive operator $T:E\rightarrow F$ is called:
\begin{itemize}
  \item \textit{positive} if $Tx \geq 0$ holds in $F$ for all $x \in E$;
  \item \textit{order bounded} if $T$ maps order bounded sets in $E$ to order bounded sets in $F$.
\end{itemize}
An orthogonally additive, order bounded operator $T:E\rightarrow F$ is called an \textit{abstract Uryson} operator.
This class of operators  was introduced and studied in 1990 by Maz\'{o}n and Segura de Le\'{o}n \cite{Maz-1,Maz-2}, and then extended to lattice-normed spaces by Kusraev and the second named author \cite{Ku-1,Ku-2,Pl-3}. Currently orthogonally additive operators are an active area of investigations \cite{Ben,Get,Gum,PP,PP-1}.
\end{definition}
For example, any linear  operator $T\in L_{+}(E,F)$ defines a positive abstract Uryson operator by $G (f) = T |f|$ for each $f \in E$.
Observe that if $T: E \to F$ is a positive orthogonally additive operator and $x \in E$ is such that $T(x) \neq 0$ then $T(-x) \neq - T(x)$, because otherwise both $T(x) \geq 0$ and $T(-x) \geq 0$ imply $T(x) = 0$. So, the above notion of positivity is far from the usual positivity of a linear operator: the only linear operator which is positive in the above sense is zero. A positive orthogonally additive operator need not be order bounded. Consider, for example, the real function $T: \mathbb R \to \mathbb R$ defined by
$$
T(x)=\begin{cases} \frac{1}{x^{2}}\,\,\,\text{if $x\neq
0$}\\
0\,\,\,\,\,\text{if $x=0$ }.\\
\end{cases}
$$

The set of all abstract Uryson operators from $E$ to $F$ we denote by $\mathcal{U}(E,F)$.
Consider some examples.
The most famous one is the nonlinear integral Uryson operator.

\begin{example}
Let $(A,\Sigma,\mu)$ and $(B,\Xi,\nu)$ be $\sigma$-finite complete measure spaces, and let $(A\times B,\mu\times\nu)$ denote the completion of their product measure space. Let $K:A\times B\times\Bbb{R}\rightarrow\Bbb{R}$ be a function satisfying the following conditions\footnote{$(C_{1})$ and $(C_{2})$ are called the Carath\'{e}odory conditions}:
\begin{enumerate}
  \item[$(C_{0})$] $K(s,t,0)=0$ for $\mu\times\nu$-almost all $(s,t)\in A\times B$;
  \item[$(C_{1})$] $K(\cdot,\cdot,r)$ is $\mu\times\nu$-measurable for all $r\in\Bbb{R}$;
  \item[$(C_{2})$] $K(s,t,\cdot)$ is continuous on $\Bbb{R}$ for $\mu\times\nu$-almost all $(s,t)\in A\times B$.
\end{enumerate}
Given $f\in L_{0}(B,\Xi,\nu)$, the function $|K(s,\cdot,f(\cdot))|$ is $\nu$-measurable  for $\mu$-almost all $s\in A$ and $h_{f}(s):=\int_{B}|K(s,t,f(t))|\,d\nu(t)$ is a well defined and $\mu$-measurable function. Since the function $h_{f}$ can be infinite on a set of positive measure, we define
$$
\text{Dom}_{B}(K):=\{f\in L_{0}(\nu):\,h_{f}\in L_{0}(\mu)\}.
$$
Then we define an operator $T:\text{Dom}_{B}(K)\rightarrow L_{0}(\mu)$ by setting
$$
(Tf)(s):=\int_{B}K(s,t,f(t))\,d\nu(t)\,\,\,\,\mu-\text{a.e.}\,\,\,\,(\star)
$$
Let $E$ and $F$ be order ideals in $L_{0}(\nu)$ and $L_{0}(\mu)$ respectively, $K$ a function satisfying $(C_{0})$-$(C_{2})$. Then $(\star)$ defines an \textit{orthogonally additive order bounded integral operator} acting from $E$ to $F$ if $E\subseteq \text{Dom}_{B}(K)$ and $T(E)\subseteq F$.
\end{example}

\begin{example} \label{Ex-1}
We consider the vector space $\mathbb R^m$, $m \in \mathbb N$ as a vector lattice with the coordinate-wise order: for any $x,y \in \mathbb R^m$ we set $x \leq y$ provided $e_i^*(x) \leq e_i^*(y)$ for all $i = 1, \ldots, m$, where $(e_i^*)_{i=1}^m$ are the coordinate functionals on $\mathbb R^m$. Let $T:\Bbb{R}^{n}\rightarrow\Bbb{R}^{m}$. Then $T\in\mathcal{U}(\Bbb{R}^{n},\Bbb{R}^{m})$ if and only if there are real functions $T_{i,j}:\Bbb{R}\rightarrow\Bbb{R}$,
$1\leq i\leq m$, $1\leq j\leq n$ satisfying $T_{i,j}(0)=0$ such that
$$
e_i^*\bigl(T(x_{1},\dots,x_{n})\bigr) = \sum_{j=1}^{n}T_{i,j}(x_{j}),
$$
In this case we write $T=(T_{i,j})$.
\end{example}

\begin{example} \label{Ex2}
Let $T:l^{2}\rightarrow\Bbb{R}$ be the operator defined by
$$
T(x_{1},\dots,x_{n},\dots) = \sum_{n\in I_{x}}n\bigl(|x_{n}|-1\bigr)
$$
where $I_{x}:=\{n\in\Bbb{N}:\,|x_{n}|\geq 1\}$. It is not difficult to check that $T$ is a positive abstract Uryson operator.
\end{example}

\begin{example} \label{Ex3}
Let $(\Omega, \Sigma, \mu)$ be a measure space, $E$ a sublattice of the vector lattice $L_0(\mu)$ of all equivalence classes of $\Sigma$-measurable functions $x: \Omega \to \mathbb R$, $F$ a vector lattice and $\nu: \Sigma \to F$ a finitely additive measure. Then the map $T: E \to F$ given by $T(x) = \nu({\rm supp} \, x)$ for any $x \in E$, is an abstract Uryson operator which is positive if and only if $\nu$ is positive.
\end{example}

Consider the following order in $\mathcal{U}(E,F):S\leq T$ whenever $T-S$ is a positive operator. Then $\mathcal{U}(E,F)$
becomes an ordered vector space.  If a vector lattice $F$ is Dedekind complete we have the following theorem.
\begin{thm}(\cite{Maz-1},Theorem~3.2)\label{th-1}.
Let $E$ and $F$ be a vector lattices, $F$ Dedekind complete. Then $\mathcal{U}(E,F)$ is a Dedekind complete vector lattice. Moreover for $S,T\in \mathcal{U}(E,F)$ and for $f\in E$ following hold
\begin{enumerate}
\item~$(T\vee S)(f):=\sup\{Tg_{1}+Sg_{2}:\,f=g_{1}\sqcup g_{2}\}$.
\item~$(T\wedge S)(f):=\inf\{Tg_{1}+Sg_{2}:\,f=g_{1}\sqcup g_{2}\}.$
\item~$(T)^{+}(f):=\sup\{Tg:\,g\sqsubseteq f\}$.
\item~$(T)^{-}(f):=-\inf\{Tg:\,g;\,\,g\sqsubseteq f\}$.
\item~$|Tf|\leq|T|(f)$.
\end{enumerate}
\end{thm}

We follow \cite{PP} in the next definition.

\begin{definition} \label{def:nar1}
Let $E,F$ be   vector lattices with  $E$ an atomless. An abstract Uryson operator  $T: E \to F$ is called
   \textit{order narrow} if  for every $e \in E$ there exists a net of decompositions $e = f_\alpha \sqcup g_\alpha$ such that $(T(f_\alpha) - T(g_\alpha) )\overset{\rm (o)}\longrightarrow 0$.
\end{definition}
It is a worth noting that linear order narrow operators were firstly  introduced by Maslyuchenko, Mykhaylyuk and  Popov in  \cite{MMP}. Lately, in setting of  lattice-normed spaces linear order narrow operators were investigated by the  author in \cite{Pl}.

{\bf Acknowledgment.}~Author is very grateful to  Mikhail Popov  for the valuable remarks and useful discussions.

\section{The Boolean algebra of  fragments of a positive Uryson operator}

Let $E,F$ be vector lattices with $F$ Dedekind complete and $T\in\mathcal{U}_{+}(E,F)$.
The purpose of this section is to describe the fragments of $T$. That is
$$
\mathcal{F}_{T}=\{S\in\mathcal{U}_{+}(E,F):\,S\wedge(T-S)=0\}.
$$
Like in the linear case we consider elementary fragments.
For a subset $\mathcal{A}$ of a vector lattice $W$ we employ  the
following notation:
\begin{gather}
\mathcal{A}^{\upharpoonleft}=\{x\in W:\exists\,\,\text{a sequence}\,\,(x_{n})\subset\mathcal{A}\,\,\text{with}\,\,x_{n}\uparrow x\};\notag \\
\mathcal{A}^{\uparrow}=\{x\in W:\exists\,\,\text{a net}\,\,(x_{\alpha})\subset\mathcal{A}\,\,\text{with}\,\,x_{\alpha}\uparrow x\}.\notag
\end{gather}
The meanings of $\mathcal{A}^{\downharpoonleft}$ and $\mathcal{A}^{\downarrow}$ are analogous. As usual, we also write
$$
\mathcal{A}^{\downarrow\uparrow}=(\mathcal{A}^{\downarrow})^{\uparrow};
\mathcal{A}^{\upharpoonleft\downarrow\uparrow}=((\mathcal{A}^{\upharpoonleft})^{\downarrow})^{\uparrow}.
$$
It is clear that $\mathcal{A}^{\downarrow\downarrow}=\mathcal{A}^{\downarrow}$, $\mathcal{A}^{\uparrow\uparrow}=\mathcal{A}^{\uparrow}$.
Consider a positive abstract Uryson operator $T:E\rightarrow F$, where $F$ is Dedekind complete.
Since $\mathcal{F}_{T}$ is a Boolean algebra, it is closed under finite suprema and
infima. In particular, all ``ups and downs'' of $\mathcal{F}_{T}$  are likewise closed under finite suprema and infima, and therefore they
are also directed upward and, respectively, downward.

\begin{definition}\label{def:adm}
A subset $D$ of a vector lattice $E$ is called a \it{lateral ideal} if the following conditions hold
\begin{enumerate}
\item~if $x\in D$ then $y\in D$ for every $y\in\mathcal{F}_{x}$;
\item~if $x,y\in D$, $x\bot y$ then $x+y\in D$.
\end{enumerate}
\end{definition}

Consider some examples.

\begin{example}\label{adm-1}
Let $E$ be a vector lattice. Every order ideal in $E$ is a lateral ideal.
\end{example}

\begin{example}\label{adm-11}
Let $E,F$ be a vector lattices and $T\in\mathcal{U}_{+}(E,F)$. Then $\mathcal{N}_{T}:=\{e\in E:\,T(e)=0\}$  is a lateral ideal.
\end{example}

The following example is important for further considerations.

\begin{lemma}(\cite{Ben}, Lemma~3.5). \label{Ex1}
Let $E$ be a vector lattice and $x\in E$. Then $\mathcal{F}_{x}$ is a lateral ideal.
\end{lemma}\label{le:02}

Let $T\in\mathcal{U}_{+}(E,F)$ and $D\subset E$ be a lateral ideal. Then for every $x\in E$, we  define a map  $\pi^{D}T:E\rightarrow F_{+}$ by the following formula
\begin{gather}
\pi^{D}T(x)=\sup\{Ty:\,y\in\mathcal{F}_{x}\cap D\}.
\end{gather}

\begin{lemma}(\cite{Ben},Lemma~3.6). \label{le:01}
Let $E,F$ be vector lattices  with $F$ Dedekind complete, $\rho\in\mathfrak{B}(F)$, $T\in\mathcal{U}_{+}(E,F)$ and $D$ be a lateral ideal.
Then $\pi^{D}T$ is a positive abstract Uryson operator and $\rho\pi^{D}T\in\mathcal{F}_{T}$.
\end{lemma}

If $D=\mathcal{F}_{x}$ then  the
operator $\pi^{D}T$ is denoted by $\pi^{x}T$. Let $F$ be a vector lattice.
Recall that a family of mutually disjoint order projections  $(\rho_{\xi})_{\xi\in\Xi}$ on $F$ is said to be {\it partition of unity} if $\bigvee\limits_{\xi\in\Xi}(\rho_{\xi})_{\xi\in\Xi}=Id_{F}$.
Any fragment of the form $\sum_{i=1}^{n}\limits\rho_{i}\pi^{x_{i}}T$, $n\in\Bbb{N}$, where $\rho_{1},\dots,\rho_{n}$ is a  finite family of mutually disjoint order projections in $F$, like in  the linear case is
called an {\it elementary} fragment  $T$. The set of all elementary fragments of $T$ we denote by $\mathcal{A}_{T}$.

For  further considerations  we need the following auxiliary proposition, which was proven by  nonstandard methods.
\begin{lemma}[\cite{Ku-Kut}, Proposition~5.2.7.2]\label{op-100}
Let $F$ be a  Dedekind complete vector lattice with a weak order unit\footnote{\,An element $u\in F_+$ is a
{\it weak order unit} if $\{u\}^{\bot\bot}=F$, i.e. except $0$ there
are no elements in $F$ which are disjoint to $u$.} $u$ and $(x_{\lambda})_{\lambda\in\Lambda}$
be an order bounded net in $F$. Then the net  $(x_{\lambda})_{\lambda\in\Lambda}$ order converges to an element $x\in F$
if and only if for every $\varepsilon>0$ there exists a partition of unity
$(\rho_{\lambda})_{\lambda\in\Lambda}$ such that
\[
\rho_{\lambda}|x_{\beta}-x|\leq\varepsilon u, \;\; \beta\geq\lambda.
\]
\end{lemma}
\par
\begin{rem}
Observe that  every Dedekind complete vector lattice is an order dense ideal in  some Dedekind complete vector lattice with a weak order unit (\cite{Vul}, Theorem~4.7.2).
\end{rem}
\begin{lemma}\label{op-10}
Let $E,F$ be vector lattices,  $F$ be  Dedekind complete and let $\mathfrak{A}$ be the set of all weak order
units in $F$.
If  operators $T,S\in\mathcal{U}_{+}(E,F)$  are disjoint, then for every  $x\in E$, $u\in \mathfrak{A}$ and $\varepsilon>0$
there exists a partition of unity $(\pi_{\xi})_{\xi\in\Xi}$ in $\mathfrak{B}(F)$ and a family
$(x_\xi)_{\xi\in\Xi}$ of fragments of $x$, such that
$$
\pi_\xi\big(Tx_\xi+S(x-x_\xi)\big)\leq\varepsilon u \;\mbox{ for all } \; \xi\in\Xi.
$$
\end{lemma}
\begin{proof}
Take any $x\in E$. Denote by $\Xi$, the set of all pairs $\xi=(y,z)\in\mathcal{F}_x\times\mathcal{F}_x$ of
mutually disjoint fragments of $x$, such that $y+z=x$.
For any $\xi=(y,x-y)\in \Xi$ put $f_\xi=Ty+S(x-y)$.
Due to  formula (2) of Theorem \ref{th-1} the disjointness
of the operators $S$ and $T$ implies $\inf\limits_{\xi\in\Xi}\{f_{\xi}\}=0$.
Denote by $\Delta$ the collection of all finite subsets of $\Xi$ ordered as usual by inclusion,
i.e. $\alpha\leq \alpha'$ iff $\alpha\subset \alpha'$.
Introduce a set $(y_{\alpha})_{\alpha\in\Delta}$ of all infima of finitely many elements of the set
$\{f_\xi\colon \xi\in \Xi\}$, i.e. if $\alpha\in\Delta$ is a finite set  $\alpha=\{\xi_{\alpha_{1}},\dots,\xi_{\alpha_{n}}\}$,
where $\xi_{\alpha_k}\in \Xi$ for $k=1,\ldots,n$, then
\[
       y_\alpha=\bigwedge\limits_{i=1}^n f_{\xi_{\alpha_i}}
\]
The set $(y_{\alpha})_{\alpha\in\Delta}$ is downwards directed and $\olim\limits_{\alpha\in\Delta}y_{\alpha}=0$.
By Proposition~\ref{op-100}, for every $\varepsilon>0$ and $u\in\mathfrak{A}$
there exists a partition of unity $(\rho_{\alpha})_{\alpha\in\Delta}$ in $\mathfrak{B}(F)$ such that
$$
\rho_{\alpha}(y_{\alpha})\leq \varepsilon u \;\mbox{ for all }\;  \alpha\in\Delta.
$$
In particular, $\rho_\alpha(f_\xi)<\varepsilon u$ if $\alpha=\xi$.
\par
Identify now $F$ with a vector sublattice of the Dedekind complete vector lattice
$C_{\infty}(Q)$ of all extended real valued continuous functions on some extremally disconnected compact
space $Q$ (more exactly with its image under some vector lattice isomorphism), where the choosen weak order
unit $u$ is mapped onto the constant function  $\bold{1}$ on $Q$ (see \cite{AbrAl}, Theorem 3.35).
Then the order projections $(\rho_{\alpha})_{\alpha\in\Delta}$
(of the above partition of unity) are the multiplication operators in the space $C_{\infty}(Q)$
generated by the characteristic functions $\bold{1}_{Q_{\alpha}}$, respectively,
where $Q_{\alpha}$ for all $\alpha\in\Delta$ are closed-open subsets of $Q$
such that $Q=\bigcup\limits_{\alpha}Q_{\alpha}$ and $Q_{\alpha}\cap Q_{\alpha'}=\emptyset$
for every $\alpha,\alpha'\in \Delta$, $\alpha\neq\alpha'$. The supremum $\sup\limits_{\alpha\in\Delta} \rho_\alpha$ is
the identity operator $I_F$.
\\
For $\alpha\in \Delta$ and $\xi\in\Xi$ define the set
\[
     A_\xi^\alpha=\{t\in Q_\alpha\colon f_\xi(t)< f_\beta(t), \, \beta\in \alpha,\, \beta\neq \xi \}
\]
and denote by $\overline{A_\xi^\alpha}$ its closure in $Q_\alpha$ and, consequently in $Q$.
So $\overline{A_\xi^\alpha}$  are closed-open subsets of $Q$
for every   $\alpha\in \Delta$, $\xi\in\Xi$  and, mutually disjoint if at least one index is different $\xi\neq \xi'$
or $\alpha\neq \alpha'$.
Denote by $\rho_\xi^\alpha$ the multiplication operator
generated by the characteristic function $\bold{1}_{\overline{A_{\xi}^{\alpha}}}$,
i.e.  $\rho_\xi^\alpha(f)= f\cdot\bold{1}_{\overline{A_\xi^\alpha}}$ for any function $f\in C_{\infty}(Q)$.
It is clear that $\rho_\xi^\alpha$ is an order projection in $C_{\infty}(Q)$ and
$\overline{A_\xi^\alpha}\subset Q_\alpha$ implies $\rho_\xi^\alpha\leq \rho_\alpha$.
Hence $\rho_\xi^\alpha(f_\xi)\leq \varepsilon u$ for every  $\xi\in\Xi$ and every $\alpha\in\Delta$.
By what has been mentioned above the order projections $\rho_\xi^\alpha$
are mutually disjoint, whenever  $\xi\neq \xi'$ or $\alpha\neq \alpha'$.
Therefore, the order projections $\pi_\xi=\sup\limits_{\alpha\in\Delta}\rho_\xi^\alpha$ and
$\pi_{\xi'}=\sup\limits_{\alpha\in\Delta}\rho_{\xi'}^\alpha$ are mutually disjoint as well.
We show that the supremum of all $\pi_\xi$ is the identity operator.
By assuming the contrary there is a nonzero order projection $\gamma$
which is disjoint to each projection $\pi_\xi$ what causes its disjointness to each
$\rho_\xi^\alpha$ and finally, $\gamma$ is disjoint to each $\rho_\alpha$.
This contradicts the fact that $(\rho_\alpha)_{\alpha\in\Delta}$ is a partition of unity.
Thus
$(\pi_\xi)_{\xi\in\Xi}$ is a partition of unity and
$$
\pi_{\xi}\big(Tx_{\xi}+S(x-x_{\xi})\big)\leq \varepsilon u \; \mbox{ for every } \; \xi\in\Xi.
$$
\end{proof}

\begin{lemma}\label{le:1}
Let $E,F,\mathfrak{A}_{F}$ be the same as in the Lemma~\ref{op-10}, $S,T\in\mathcal{U}_{+}(E,F)$.
If $S\bot T$, then  for every $x\in E$, $\varepsilon>0$, $\bold{1}\in\mathfrak{A}_{F}$
there exists a partition of unity   $(\rho_{\xi})_{\xi\in\Xi}$ in $\mathfrak{B}(F)$,
and a family $(x_{\xi})_{\xi\in\Xi}$ of fragments of $x$ such that
$\rho_{\xi}\pi^{x_{\xi}}T(x)\leq\varepsilon\bold{1}$ and
$\rho_{\xi}(S-\rho_{\xi}\pi^{x_{\xi}}S)x\leq\varepsilon\bold{1}$ for every $\xi\in\Xi$.
\end{lemma}

\begin{proof}
Observe that for every $y\in \mathcal{F}_{x}$, $x\in E$ we have $\pi^{y}Tx=Ty$. Fix a weak order unit $\bold{1}$ and $\varepsilon>0$. By  Lemma~\ref{op-10} there exist a partition of unity   $(\rho_{\xi})_{\xi\in\Xi}$ in $F$,
and a family $(x_{\xi})_{\xi\in\Xi}$ of fragments of $x$ such that
$$
\rho_\xi\big(Tx_\xi+S(x-x_\xi)\big)\leq\varepsilon u \;\mbox{ for all } \; \xi\in\Xi.
$$
Consequently,
$\rho_{\xi}Tx_{\xi}=\rho_{\xi}\pi^{x_{\xi}}Tx\leq\varepsilon\bold{1}$ and
$$
\rho_{\xi}S(x-x_{\xi})=\rho_{\xi}Sx-\rho_{\xi}Sx_{\xi}=
\rho_{\xi}(S-\rho_{\xi}\pi^{x_{\xi}}S)x\leq\varepsilon\bold{1}.
$$
\end{proof}

\begin{lemma}\label{le:2}
Let $E,F,\mathfrak{A}_{F}$ be the same as in the Lemma~\ref{op-10}, $T\in\mathcal{U}_{+}(E,F)$.
If $S\in\mathcal{F}_{T}$ then for every $x\in E$, $\varepsilon>0$, $\bold{1}\in\mathfrak{A}_{F}$
there exists a partition of unity   $(\rho_{\xi})_{\xi\in\Xi}$ in $\mathfrak{B}(F)$,
and a family $(x_{\xi})_{\xi\in\Xi}$ of fragments of $x$, such that
$\rho_{\xi}|S-\rho_{\xi}\pi^{x_{\xi}}T|x\leq\varepsilon\bold{1}$ for every $\xi\in\Xi$.
\end{lemma}

\begin{proof}
Using Lemma~\ref{le:1}   we have
$$
\rho_{\xi}|S-\rho_{\xi}\pi^{x_{\xi}}T|x\leq\rho_{\xi}|S-\rho_{\xi}\pi^{x_{\xi}}S|x+
\rho_{\xi}|\rho_{\xi}\pi^{x_{\xi}}S-\rho_{\xi}\pi^{x_{\xi}}T|x=
$$
$$
=\rho_{\xi}|S-\rho_{\xi}\pi^{x_{\xi}}S|x+
\rho_{\xi}|\rho_{\xi}\pi^{x_{\xi}}(T-S)|x\leq\varepsilon\bold{1}.
$$
\end{proof}

\begin{lemma}\label{le:3}
Let $E,F$ be the same as in Lemma~\ref{le:1},  $T\in\mathcal{U}_{+}(E,F)$ and $S\in\mathcal{F}_{T}$. Then
\begin{enumerate}
\item for every $x\in E$, $\varepsilon>0$, $\bold{1}\in\mathfrak{A}_{F}$ there exists
$G_{x}\in\mathcal{A}_{T}^{\uparrow}$, so that $|S-G_{x}|x\leq\varepsilon\bold{1}$;
\item for every $x\in E$  there exists
$R_{x}\in\mathcal{A}_{T}^{\uparrow\downharpoonleft}$, so that $|S-R_{x}|x=0$.
\end{enumerate}
\end{lemma}

\begin{proof}
Let us to prove $(1)$. By Lemma~\ref{le:2} there exists a partition of unity   $(\rho_{\xi})_{\xi\in\Xi}$ in $\mathfrak{B}(F)$,
and a family $(x_{\xi})_{\xi\in\Xi}$
of fragments of  $x$ such that
$\rho_{\xi}|S-\rho_{\xi}\pi^{x_{\xi}}T|x\leq\varepsilon\bold{1}$ for  $\xi\in\Xi$.
By $\Delta$ we denote the system of all finite subsets of $\Xi$. It is  ordered by inclusion.
Surely, $\Delta$ is a directed set. For every $\theta\in\Delta$ set
$G_{\theta}=\sum\limits_{\theta\in\Delta}\rho_{\xi}\pi^{x_{\xi}}T$. The net $(G_{\theta})_{\theta\in\Delta}$ is increasing.
Let $G_{x}=\sup(G_{\theta})_{\theta\in\Delta}$. Then $G_{x}\in\mathcal{A}_{T}^{\uparrow}$ and we may write
$$
\rho_{\xi}|S-G_{\theta}|x=
\rho_{\xi}|S-\sum\limits_{\theta\in\Delta}\rho_{\xi}\pi^{x_{\xi}}T|x\leq\varepsilon\bold{1}
$$
for every $\xi\in\Xi$ and every $\theta\geq\{\xi\}$. Therefore $\rho_{\xi}|S-G_{x}|x\leq\varepsilon\bold{1}$ for every
$\xi\in\Xi$ and $|S-G_{x}|x\leq\varepsilon\bold{1}$.

Now we prove $(2)$. Fix any $\bold{1}\in\mathfrak{A}_{F}$. For $\varepsilon_{n}=\frac{1}{2^{n}}$ there exists
$G_{x}^{n}\in\mathcal{A}_{T}$ such that $|S-G_{x}^{n}|x\leq\frac{1}{2^{n}}\bold{1}$. Let $C_{x}^{k}=\bigvee\limits_{n=k}^{\infty}G_{x}^{n}$ and
$C_{x}^{k,i}=\bigvee\limits_{n=k}^{n=k+i}G_{x}^{n}$. Since $\mathcal{A}_{T}$ is a subalgebra of $\mathcal{F}_{T}$, one has $C_{x}^{k,i}\in\mathcal{A}_{T}^{\uparrow}$  and
$C_{x}^{k,i}\uparrow C_{x}^{k}\in\mathcal{A}_{T}^{\uparrow\upharpoonleft}=\mathcal{A}_{T}^{\uparrow}$. Then we have
$$
\Big|S-C_{x}^{k,i}\Big|x=\Big|S-\bigvee\limits_{n=k}^{n=k+i}G_{x}^{n}\Big|x=
\Big|\bigwedge\limits_{n=k}^{n=k+i}(S-G_{x}^{n})\Big|x\leq
$$
$$
\leq\sum\limits_{n=k}^{n=k+i}\Big|S-G_{x}^{n}\Big|x\leq\sum\limits_{n=k}^{\infty}\frac{1}{2^{n}}\bold{1}\leq
\frac{1}{2^{k-1}}\bold{1}.
$$
So we may write  $|S-C_{x}^{k}|\leq\frac{1}{2^{k-1}}\bold{1}$. The sequence $(C_{x}^{k})$ is decreasing. Let $R_{x}=\inf{C_{x}^{k}}$. Then $R_{x}\in\mathcal{A}_{T}^{\uparrow\downharpoonleft}$ and $|S-R_{x}|x=0$.
\end{proof}

\begin{rem}\label{rem:31}
Observe that $R_{x}y=0$ for every $y$ such that $\mathcal{F}_{x}\cap\mathcal{F}_{y}=0$.
Moreover, if $y\in\mathcal{F}_{x}$ and $|S-R_{x}|x=0$ we can write $0\leq|S-R_{x}|y\leq|S-R_{x}|x=0$, and therefore $|S-R_{x}|y=0$
for every $y\in\mathcal{F}_{x}$.
\end{rem}

\begin{lemma}\label{le:4}
Let $E,F$ be the same as in Lemma~\ref{le:3},  $T\in\mathcal{U}_{+}(E,F)$, $x\in E$ and
$S\in\mathcal{F}_{T}$. Then there exists a $G\in\mathcal{A}_{T}^{\uparrow\downarrow}$ such that:
$$
0\leq G\leq S \,\,\text{and}\,\,  Gx=Sx.
$$
\end{lemma}
\begin{proof}
Fix $x\in E$ and let
$$
W:=\{R\in\mathcal{A}_{T}^{\uparrow\downharpoonleft}:\,|S-R|x=0\}.
$$
By  Lemma~\ref{le:3} the set $W$ is nonempty, and an easy argument shows that
$W$ is directed downward. Let $G=\inf\{W\}$. Clearly,
$G\in\mathcal{A}_{T}^{\uparrow\downharpoonleft\downarrow}=\mathcal{A}_{T}^{\uparrow\downarrow}$, and
hence $|S-G|x=0$
We claim that $0\leq G\leq S$. By  Remark~\ref{rem:31} $Gz=0$ for every $z\in E$, such that $\mathcal{F}_{z}\cap\mathcal{F}_{x}=0$ and we must  prove $(G-S)^{+}y=0$ for every $y\in\mathcal{F}_{x}$.
Now we may write
$$
(G-S)^{+}y\leq |R_{x}-S|y=|S-R_{x}|y=0,
$$
where $y\in\mathcal{F}_{x}$ and $R_{x}$ is a some element of $W$.
\end{proof}
The following theorem is the first main result of the article.
\begin{thm}\label{frag-2}
Let $E,F$ be vector lattices, $F$ Dedekind complete, $T\in\mathcal{U}_{+}(E,F)$ and $S\in\mathcal{F}_{T}$.
Then $S\in\mathcal{A}_{T}^{\uparrow\downarrow\uparrow}$.
\end{thm}

\begin{proof}
Let $S\in\mathcal{F}_{T}$ be fixed, and let
$$
W=\{R\in\mathcal{A}_{T}^{\uparrow\downarrow}:\,0\leq R\leq S\}
$$
Clearly, $W$ is a directed set, and by Lemma~\ref{le:4} we know that $W\neq\emptyset$. Let $G=\sup\{W\}$,
and remark that $0\leq G\leq S$. On the
other hand, if $x\in E$ is an  arbitrary element of $E$, by  Lemma~\ref{le:4} there exists some $R\in W$, such that
$0\leq R\leq G\leq S$ and $Rx=Sx$. Thus $G=S$, $S\in\mathcal{A}_{T}^{\uparrow\downarrow\uparrow}$ and $\mathcal{F}_{T}=\mathcal{A}_{T}^{\uparrow\downarrow\uparrow}$.
\end{proof}
Remark that for linear positive operators the same theorem and its  modifications were proved by  de Pagter, Aliprantis and Burkinshaw, Kusraev and Strizhevski   in \cite{Al-1,KS,Pag}.

\section{Domination problem for abstract Uryson narrow operators}
\label{sec5}

In this section we consider a domination problem for narrow abstract Uryson operators. In the classical sense, the domination problem can be stated as follows. Let $E$, $F$ be vector lattices, $S,T: E \to F$ linear operators with $0 \leq S \leq T$. Let $\mathcal P$ be some property of linear operators $R: E \to F$, so that $\mathcal P(R)$ means that $R$ possesses $\mathcal P$. Does $\mathcal P(T)$ imply $\mathcal P(S)$?

Let $E$ be a vector lattice and $x\in E_{+}$. The order ideal generated by $x$ we denote by $E_{x}$. The following theorem is an important tool for further considerations.
\begin{thm}(Freudenthal Spectral Theorem)(\cite{Al}, Theorem~2.8).\label{Fr}
Let $E$ be a vector lattice with the principal projection property
and let $x\in E_{+}$. Then for every $y\in E_{x}$
there exists a sequence $(u_{n})$ of $x$-step functions satisfying
$0\leq y-u_{n}\leq\frac{1}{n}x$ for each $n$ and $u_{n}\uparrow y$.
\end{thm}
The next theorem is  the second main result of the article.
\begin{thm} \label{thm:dom}
Let $E,F$ be vector lattices,  $E$ atomless and with the principal projection property,  $F$ be  Dedekind complete, and  $T\in\mathcal{U}_{+}(E,F)$ be an order narrow operator. Then every   operator $S\in\mathcal{U}_{+}(E,F)$, such that $0\leq S\leq T$ is order narrow as well.
\end{thm}
For the proof we need an some auxiliary result.
Let $E,F$ be vector lattices, a family of operators $\{T_{1},\dots, T_{n}\}\subset\mathcal{U}(E,F)$ is said to have
\textit{pairwise disjoint supports} if there exists a family of pairwise disjoint bands $E_{1},\dots,E_{n}\subset E$, such that $T_{i}x=0$ for every $x\in E_{i}^{\bot}$, $i\in\{1,\dots,n\}$.

\begin{lemma} \label{l:sum}
Let $E,F$ be vector lattices,  $E$ atomless and with the projection property,  $F$ be  Dedekind complete,  and $\{T_{1},\dots, T_{n}\}\subset\mathcal{U}(E,F)$ be a family of order narrow operators with pairwise disjoint  supports. Then  $T=\sum\limits_{i=1}^{n}T_{i}$ is  an order narrow operator as well.
\end{lemma}
\begin{proof}
Fix an arbitrary element $e\in E$. Let $\rho_{i}$,  be a band projection to the band $E_{i}$, $i\in\{1,\dots,n\}$, and $\zeta=Id-\bigvee\limits_{i=1}^{n}\rho_i$. Then we may write $e=h\sqcup\bigsqcup\limits_{i=1}^{n}e_{i}$, where $e_{i}=\rho_{i}e$, $i\in\{1,\dots,n\}$ and $h=\zeta e$. By our assumption for every $e_{i}$, $i\in\{1,\dots,n\}$ there exists a net of decompositions $e_{i} = e_{i1}^{\alpha} \sqcup e_{i2}^{\alpha}$ such that $(T_{i}(e_{i1}^{\alpha}) - T_{i}(e_{i2}^{\alpha}) )\overset{\rm (o)}\longrightarrow 0$. Let $f_{\alpha}=\bigsqcup\limits_{i=1}^{n}e_{i1}^{\alpha}$ and $g_{\alpha}=\bigsqcup\limits_{i=1}^{n}e_{i2}^{\alpha}$. Now we have
\begin{align*}
|T(h+f_{\alpha})-T(g_{\alpha})|=
\Big|\sum\limits_{i=1}^{n}T_{i}\Big(h\sqcup\bigsqcup\limits_{j=1}^{n}e_{j1}^{\alpha}\Big)-
\sum\limits_{i=1}^{n}T_{i}\Big(\bigsqcup\limits_{j=1}^{n}e_{j2}^{\alpha}\Big)\Big|= \\
\Big|\sum\limits_{i=1}^{n}T_{i}(e_{i1}^{\alpha})-\sum\limits_{i=1}^{n}T_{i}(e_{i2}^{\alpha})\Big|=
\Big|\sum\limits_{i=1}^{n}(T_{i}(e_{i1}^{\alpha})-T_{i}(e_{i2}^{\alpha})\Big|\leq \\
\sum\limits_{i=1}^{n}\Big|T_{i}(e_{i1}^{\alpha})-T_{i}(e_{i2}^{\alpha})\Big|\overset{\rm (o)}\longrightarrow 0.
\end{align*}
Thus $(h\sqcup f_{\alpha}) \sqcup g_{\alpha}=e$ is the desired net of decompositions.
\end{proof}

\begin{lemma}\label{le:12}
Let $E,F$ be the same as in the Theorem~\ref{thm:dom}, $x_{1},x_{2}\in E$ and $x_{1}\bot x_{2}$. Then  $\pi^{x_{1}+ x_{2}}T=\pi^{x_{1}}T+\pi^{x_{2}}T$ for every $T\in\mathcal{U}_{+}(E,F)$.
\end{lemma}
\begin{proof}
Fix an arbitrary element $x\in E$. Then for every $y\in\mathcal{F}_{x}$ so that  $y\sqsubseteq (x_{1}+x_{2})$,  we have $y=y_{1}\sqcup y_{2}$, $y_{i}\sqsubseteq x_{i}$, $i\in\{1,2\}$ and the following inequalities hold
\begin{align*}
Ty=Ty_{1}+Ty_{2}\leq\pi^{x_{1}}Tx+\pi^{x_{2}}Tx;\\
\pi^{x_{1}+x_{2}}Tx\leq\pi^{x_{1}}Tx+\pi^{x_{2}}Tx.
\end{align*}
On the other hand for every $y_{i}\sqsubseteq x_{i}$, $y_{i}\sqsubseteq x$, $i\in\{1,2\}$ we may write
\begin{align*}
Ty_{1}+Ty_{2}=T(y_{1}+y_{2})\leq\pi^{x_{1}+x_{2}}Tx;\\
\pi^{x_{1}}Tx+\pi^{x_{2}}Tx\leq\pi^{x_{1}+x_{2}}Tx.
\end{align*}
\end{proof}
\begin{proof}[Proof of Theorem~\ref{thm:dom}]
Let $T\in\mathcal{U}_{+}(E,F)$ be an order narrow operator, and $x\in E$. Firstly we prove that operator $\rho\pi^{x}T$ is also order narrow. Fix an arbitrary element $e\in E$. By our assumption there exists a net of decompositions $e = f_{\alpha}\sqcup g_{\alpha}$ such that $|T(f_{\alpha}) - T(g_{\alpha})|\leq\eta_{\alpha}$, $(\eta_{\alpha})\subset F_{+}$ and $(\eta_{\alpha})\downarrow 0$. Remark that $D=\{y\sqsubseteq e:\,y\in\mathcal{F}_{x}\}$ is a directed set and by definition of the operator $\pi^{x}T$  there exists a net $(y_{\alpha})\subset D$ so that
$$
|\pi^{x}Te-Ty_{\alpha}|=|\pi^{x}Te-\pi^{x}Ty_{\alpha}|=|\pi^{x}T(e-y_{\alpha})|\leq\xi_{\alpha}
$$
for some decreasing net $(\xi_{\alpha})\subset F_{+}$, $\inf\limits_{\alpha}\xi_{\alpha}=0$. By our assumption there exists a net of decompositions $y_{\alpha} = f_{\alpha} \sqcup g_{\alpha}$ such that $|T(f_{\alpha}) - T(g_{\alpha})|\leq\eta_{\alpha}$, $(\eta_{\alpha})\subset F_{+}$ and $(\eta_{\alpha})\downarrow 0$. Then we may write
\begin{align*}
|\pi^{x}T((e-y_{\alpha})\sqcup f_{\alpha}))-\pi^{x}T(g_{\alpha})|=\\
|\pi^{x}T(e-y_{\alpha})+\pi^{x}Tf_{\alpha}-\pi^{x}Tg_{\alpha}|=\\
|\pi^{x}T(e-y_{\alpha})+Tf_{\alpha}-Tg_{\alpha}|\leq\\
|\pi^{x}T(e-y_{\alpha})|+|Tf_{\alpha}-Tg_{\alpha}|\leq\xi_{\alpha}+\eta_{\alpha}\overset{\rm (o)}\longrightarrow 0.
\end{align*}
So $((e-y_{\alpha})\sqcup f_{\alpha}))\sqcup g_{\alpha}=e$ is a desired net of decompositions. It is clear that operator $\rho\pi^{x}T$ is order narrow as well.
Secondly, take the operator $R=\sum\limits_{i=1}^{n}\rho_{i}\pi^{x_{i}}T$, where $x_{1},\dots, x_{n}$ are fragments of a some element $x\in E$ and $\rho_{1},\dots,\rho_{n}$ are mutually disjoint. By  Lemma~\ref{le:12}, we may assume that all fragments $x_{1},\dots, x_{n}$ are mutually disjoint. Then applying  Lemma~\ref{l:sum} we prove that $R$ is an order narrow operator. Now, let  $(R_{\xi})_{\xi\in\Xi}\subset\mathcal{U}_{+}(E,F)$ be an increasing (decreasing) net of order narrow operators and $S=\sup\limits_{\xi}R_{\xi}$ ($S=\inf\limits_{\xi}R_{\xi}$). This meant that there exists a decreasing net $(G_{\xi})_{\xi\in\Xi}\subset\mathcal{U}_{+}(E,F)$, so that
$\inf\limits_{\xi}G_{\xi}=0$ and
\begin{align*}
|Se-R_{\xi}e|=|(S-R_{\xi})e|\leq|S-R_{\xi}|e\leq G_{\xi}e
\end{align*}
for every $e\in E$. Let us  show that $S$ is also order narrow. Indeed, fix an arbitrary element $e\in E$ and write
\begin{align*}
|Sf_{\alpha}-Sg_{\alpha}|=|Sf_{\alpha}-R_{\xi}f_{\alpha}+R_{\xi}f_{\alpha}-
R_{\xi}g_{\alpha}+R_{\xi}g_{\alpha}-Sg_{\alpha}|\leq \\
|Sf_{\alpha}-R_{\xi}f_{\alpha}|+|R_{\xi}f_{\alpha}-
R_{\xi}g_{\alpha}|+|Sg_{\alpha}-R_{\xi}g_{\alpha}|\leq \\
G_{\xi}f_{\alpha}+|R_{\xi}f_{\alpha}-
R_{\xi}g_{\alpha}|+G_{\xi}g_{\alpha}\leq \\
G_{\xi}e+|R_{\xi}f_{\alpha}-
R_{\xi}g_{\alpha}|+G_{\xi}e\overset{\rm (o)}\longrightarrow 0.
\end{align*}
By Theorem~\ref{frag-2} we have that $\mathcal{F}_{T}=\mathfrak{A}_{T}^{\uparrow\downarrow\uparrow}$ and  applying  this equality we obtain that every fragment of an order narrow operator $T$ is also  order narrow.  Finally take an arbitrary operator $S\in\mathcal{U}(E,F)$, so that $0\leq S\leq kT$, $k\in\Bbb{R}_{+}$. By
Theorem~\ref{Fr} there exists a sequence $R_{n}$ of $T$-step positive abstract Uryson operators $R_{n}=\sum\limits_{i=1}^{k_{n}}\lambda_{i}C_{i}$, where $\lambda_{i}>0$ for $i\in\{1,\dots, n\}$ and the operators $C_{1},\dots, C_{k_{n}}$ are pairwise disjoint fragments of $T$ such that   so that $|S(e)-R_{n}(e)|\leq\frac{1}{n}T(e)$ for every $e\in E$. Dividing by $\max\{\lambda_{i}:i=1,\dots,k_{n}\}$ we may assume that $\lambda_{i}\leq 1$ for every $i\in\{1,\dots,k_{n}\}$ and therefore $0\leq R_{n}=\sum\limits_{i=1}^{k_{n}}\lambda_{i}C_{i}=\bigvee\limits_{i=1}^{k_{n}}\lambda_{i}C_{i}\leq T$ is a fragment of the operator $T$ for every $n\in\Bbb{N}$. Thus $R_{n}$ is an order narrow operator for every $n\in\Bbb{N}$. Finally, using the same arguments as above, we obtain that $S$ is an order narrow operator.
\end{proof}
Remark that for linear positive operators the similar theorem  was  proved by  Flores and Ruiz    in \cite{FR}.

\end{document}